\documentclass[11pt]{article}

\renewenvironment{thebibliography}[1]{
  \begin{oldthebibliography}{#1}
    \setlength{\itemsep}{0.4em}
    \setlength{\parskip}{0.2em}
}
{
  \end{oldthebibliography}
}
\usepackage[latin9]{inputenc}  
\usepackage{amsmath}
\usepackage{amsthm}
\usepackage{amssymb}
\usepackage{setspace}
\usepackage[unicode=true,
 bookmarks=false,
 breaklinks=false,pdfborder={0 0 1},colorlinks=false]
 {hyperref}

\makeatletter
\theoremstyle{plain}
\newtheorem{thm}{\protect\theoremname}
  \theoremstyle{plain}
  
  \theoremstyle{plain}
  \newtheorem{lem}[thm]{\protect\lemmaname}

\usepackage{amsfonts}
\usepackage{latexsym}\usepackage{xcolor}\usepackage{epsfig}\usepackage[textsize=scriptsize,colorinlistoftodos]{todonotes}
\usepackage[mathscr]{euscript}
\let\oldcite=\cite 
\renewcommand{\cite}[1]{\ifthenelse{\equal{#1}{need}}{\todo{citation needed}}{\oldcite{#1}}}   

\def\qed{\ifvmode\mbox{ }\else\unskip\fi\hskip 1em plus 10fill$\Box$}
\input{epsf}

\def\Ddots{\mathinner{\mkern1mu\raise\p@
\vbox{\kern7\p@\hbox{.}}\mkern2mu
\raise4\p@\hbox{.}\mkern2mu\raise7\p@\hbox{.}\mkern1mu}}

\author{
Jacob Fox\thanks{Department of Mathematics, Stanford University, Stanford, CA 94305. Email: {\tt jacobfox@stanford.edu}. Research supported by a Packard Fellowship, by NSF Career Award DMS-1352121 and by an Alfred P. Sloan Fellowship.}
\and
Huy Tuan Pham \thanks{Stanford University, Stanford, CA 94305. Email: {\tt huypham@stanford.edu}. Research supported by the Stanford Undergraduate Research Institute in Mathematics (SURIM).}}
\title{\vspace{-0.7cm} {Popular progression differences in vector spaces}}

  \providecommand{\corollaryname}{Corollary}
  \providecommand{\lemmaname}{Lemma}
\providecommand{\theoremname}{Theorem}

  \providecommand{\corollaryname}{Corollary}
  \providecommand{\lemmaname}{Lemma}
\providecommand{\theoremname}{Theorem}

  \providecommand{\corollaryname}{Corollary}
  \providecommand{\lemmaname}{Lemma}
\providecommand{\theoremname}{Theorem}

  \providecommand{\corollaryname}{Corollary}
  \providecommand{\lemmaname}{Lemma}
\providecommand{\theoremname}{Theorem}

  \providecommand{\corollaryname}{Corollary}
  \providecommand{\lemmaname}{Lemma}
\providecommand{\theoremname}{Theorem}

  \providecommand{\corollaryname}{Corollary}
  \providecommand{\lemmaname}{Lemma}
\providecommand{\theoremname}{Theorem}

  \providecommand{\corollaryname}{Corollary}
  \providecommand{\lemmaname}{Lemma}
\providecommand{\theoremname}{Theorem}

  \providecommand{\corollaryname}{Corollary}
  \providecommand{\lemmaname}{Lemma}
  \providecommand{\theoremname}{Theorem}
\providecommand{\theoremname}{Theorem}

\makeatother

  \providecommand{\corollaryname}{Corollary}
  \providecommand{\lemmaname}{Lemma}
\providecommand{\theoremname}{Theorem}

\begin{document}
\maketitle 
\begin{abstract}


Green proved an arithmetic analogue of Szemer\'edi's celebrated regularity lemma and used it to verify a conjecture of Bergelson, Host, and Kra 
which sharpens Roth's theorem on three-term arithmetic progressions in dense sets. It shows that for every subset of $\mathbb{F}_p^n$ with $n$ sufficiently large, the density of three-term arithmetic progressions with some nonzero common difference is at least the random bound (the cube of the set density) up to an additive $\epsilon$.  For a fixed odd prime $p$, we prove that the required dimension grows as an exponential tower of $p$'s of height $\Theta(\log(1/\epsilon))$. This improves both the lower and upper bound, and is the first example of a result where a tower-type bound coming from applying a regularity lemma is shown to be necessary.

\end{abstract}

\section{Introduction}

There is a long history in number theory of studying density conditions
which guarantee three-term arithmetic progressions in abelian groups.
Addressing a problem of Erd\H{o}s and Tur\'an \cite{ErTu} from 1936,
Roth \cite{R53} in 1953 used Fourier analysis to prove that for each
$\alpha>0$ there is $N(\alpha)$ such that for every $N\geq N(\alpha)$,
every subset of $\mathbb{Z}_{N}$ of density at least $\alpha$ has
a three-term arithmetic progression; see \cite{Bl,S11} for
the best known quantitative estimates. Brown and Buhler \cite{BB82}
in 1982 proved the analogous result in vector spaces over $\mathbb{F}_{3}$.
Better estimating the quantitative bound for this problem is known
as the \textit{cap set problem}. It is a prominent problem in part
because of its close connections to other important problems such
as creating efficient algorithms for matrix multiplication \cite{BCCGNSU}
and the sunflower problem \cite{ASU,NS}. Meshulam \cite{M95} gave a
better quantitative bound on the cap set problem by developing the
finite field analogue of Roth's Fourier analytic proof. A further
quantitative improvement by Bateman and Katz \cite{BK12} uses the
additive structure of the large Fourier coefficients. Recently, there
was a breakthrough on the cap set problem by Croot, Lev, and Pach
\cite{CLP} using the polynomial method. Building on this breakthrough,
Ellenberg and Gijswijt \cite{EG} proved that any subset of $\mathbb{F}_{3}^{n}$
with no three-term arithmetic progression has at most $O(2.756^{n})$
elements. In the other direction, Edel \cite{Edel} constructed a
subset of $\mathbb{F}_{3}^{n}$ with no three-term arithmetic progression
which has $\Omega(2.217^{n})$ elements.

Szemer\'edi's regularity lemma \cite{Sz76} and its variants are some
of the most powerful tools in combinatorics. Szemer\'edi used an early
version of the regularity lemma to prove his celebrated theorem
extending Roth's theorem to arithmetic progressions of any given length.
The regularity lemma roughly says that the vertex set of every graph
can be partitioned into a small number of parts such that for most of the
pairs of parts, the induced bipartite subgraph between the pair is
pseudorandom.

Observe that a random subset $A$ of $\mathbb{F}_{p}^{n}$ of density
$\alpha$ almost surely satisfies that for every nonzero $d\in\mathbb{F}_{p}^{n}$,
the density of three-term arithmetic progressions with common difference
$d$ that are in $A$ is close to $\alpha^{3}$. For $p=3$, Edel's
construction \cite{Edel} shows that there are sets with density $\alpha$
whose density of three-term arithmetic progressions is $O(\alpha^{4.63})$,
which is substantially smaller than the random bound of $\alpha^{3}$.
However, Green \cite{Green05}, answering a question of Bergelson,
Host, and Kra \cite{BHKR}, established an arithmetic regularity lemma
and used it to prove the following theorem showing that there is a
nonzero $d$ for which the density of three-term arithmetic progressions
with common difference $d$ is at least almost $\alpha^{3}$, the
random bound. 
\begin{thm}[Green \cite{Green05}]
\label{thm:Original theorem} For each $\epsilon>0$ and prime $p$
there is a least positive integer $n_{p}(\epsilon)$ such that the
following holds. For each $n\geq n_{p}(\epsilon)$ and every subset
$A$ of $\mathbb{F}_{p}^{n}$ with density $\alpha$, there is a nonzero
$d$ in $\mathbb{F}_{p}^{n}$ such that the density of three-term
arithmetic progressions with common difference $d$ that are in $A$
is at least $\alpha^{3}-\epsilon$. 
\end{thm}
Green and Tao \cite{Green05a,GrTa10} proved an analogous
result for arithmetic progressions of length four, and an earlier
construction of Ruzsa \cite{BHKR} shows that it does not hold for
longer lengths. Tao \cite{Taoblogspectral} later observed that Green's arithmetic regularity lemma essentially follows from the spectral proof of Szemer\'edi's regularity lemma applied to the Cayley graph of the set. 

Despite being widely used, a major drawback in applying Szemer\'edi's regularity lemma or Green's arithmetic regularity lemma is that the bound on the number of parts is enormous as a function of the approximation parameter $\epsilon$, namely a tower of twos of height $\epsilon^{-O(1)}$. This gives seemingly weak bounds for the various applications. That
such a huge bound is necessary in Szemer\'edi's regularity lemma was
proved by Gowers \cite{Gow97} using a probabilistic construction.
Later, Conlon and the first author \cite{CF12}, Moshkovitz and Shapira \cite{MoSh},
and the first author and Lov\'asz \cite{FL16+} gave improvements on various aspects.
Similarly, Green \cite{Green05} showed that a tower-type bound is
neccesary in the arithmetic regularity lemma, and Hosseini, Lovett,
Moshkovitz and Shapira \cite{HLMS} improved the lower bound
on the tower height to $\epsilon^{-\Omega(1)}$.

While a tower-type bound is known to be necessary for these regularity
lemmas, for the various applications, it was generally believed that
much better bounds should hold, and this would be shown by developing
alternative techniques to the regularity method. This belief turned into a major program and 
has been quite successful for many such applications. A few examples include Gowers'
influential proof of Szemer\'edi's theorem \cite{Gow01} which introduced
higher Fourier analysis, new bounds on Ramsey numbers of sparse graphs
(see \cite{CFS,Lee} and their references) using the greedy
embedding method or dependent random choice, certain applications
in extremal graph theory using the absorption method \cite{LSS,RRS}, the first author's proof of the graph removal lemma \cite{F11},
and the recent tight, polynomial bound on the arithmetic removal lemma in vector spaces over a fixed finite field
by the first author and Lov\'asz \cite{FL}. Prior to this paper, no
known application required the tower-type bound coming from applying the
regularity lemma.


In this paper, we obtain an essentially tight bound on $n_p(\epsilon)$ in Theorem \ref{thm:Original theorem}. We prove lower and upper bounds
on the necessary dimension which grow as an exponential tower 
of height $\Theta(\log(1/\epsilon))$. This is the first
application of a regularity lemma for which the tower-type bound 
that comes from using a regularity lemma is necessary. 
\begin{thm}
\label{main1} Let $p$ be an odd prime. Recall $n_{p}(\epsilon)$ is the
least positive integer such that the following holds. For each $n\geq n_{p}(\epsilon)$
and every subset $A$ of $\mathbb{F}_{p}^{n}$ with density $\alpha$,
there is a nonzero $d \in \mathbb{F}_{p}^{n}$ such that the density
of three-term arithmetic progressions with common difference $d$ in $A$ is at least $\alpha^{3}-\epsilon$. For $\epsilon \leq 2^{-161}p^{-8}$, we have $n_{p}(\epsilon)$
is bounded from below and above by an exponential tower of $p$'s of
height $\Theta(\log(1/\epsilon))$. 
\end{thm}

It is also interesting to determine the bound on the dimension $n$ as a function of both the set density $\alpha$ and $\epsilon$. Let $n_p(\alpha,\epsilon)$ be the least positive integer such that for every $n \geq n_p(\alpha,\epsilon)$ and every subset of $\mathbb{F}_p^n$ of density at least $\alpha$, there is a nonzero $d$ in $\mathbb{F}_p^n$  such that the density of three-term
arithmetic progressions with common difference $d$ that are in $A$
is at least $\alpha^3-\epsilon$. Observe that $n_p(\epsilon)=\max_{\alpha} n_p(\alpha,\epsilon)$, that is, the above theorem determines the minimum dimension that works for all set densities. In a subsequent work \cite{FP8}, we determine for $p \geq 19$ the tower height in $n_p(\alpha,\epsilon)$ up to an absolute constant factor and an additive term depending only on $p$. The answer turns out to have different forms in different ranges of $\alpha$ and $\epsilon$, and requires additional ingredients in the proofs. We also discuss in \cite{FP8} many related problems. 

\vspace{0.1cm}
\noindent \textbf{Organization.} In Section \ref{sectionupperbound}, we prove the upper
bound in Theorem \ref{main1}. One important tool for this section is an arithmetic weak regularity lemma, and a corresponding counting lemma. While they are crucial in the proof of the upper bound, they have essentially appeared in \cite{Green06}. We include their proofs in Section \ref{sectionweak} for completeness. In Section \ref{sectionlowerbound},
we prove the lower bound in Theorem \ref{main1}.

For the sake of clarity of presentation, we omit floor and
ceiling signs where they are not crucial. When we write
$\log$ without specifying the base, we implicitly assume that the
logarithm is taken base $2$. We often use $3$-AP as shorthand for three-term
arithmetic progression.   

\section{Weak regularity lemma and counting lemma}

\label{sectionweak}

The weak regularity lemma and counting lemma are crucial in the proof of the upper bound in Theorem \ref{main1} to approximate the density of $3$-APs of a function by that of a simpler function. They have essentially appeared under various forms in \cite{Green06}. We include their proofs here for completeness. 

Let $G=\mathbb{F}_{p}^{n}$. The \textit{density}
of a function $f:G\to[0,1]$ is defined to be $\mathbb{E}_{x\in G}[f(x)]$.
For each subspace $H$ of $G$, denote $f_{H}(x)=\mathbb{E}_{y\in H+x}[f(y)]$.
It is easy to see that the density of $f_{H}$ is the same as the
density of $f$. We will also refer to $f_{H}$ as the average function
of $f$ with respect to $H$. 

The weak regularity is defined via Fourier uniformity. For a group $G$, let $\widehat{G}$ be the group of characters $\chi:G\to\mathbb{C}$.
For $G=\mathbb{F}_{p}^{n}$, the group of characters $\widehat{G}$ is isomorphic to $G$ and each character $\chi\in\widehat{G}$ can be identified with an element $t_{\chi}\in G$ such that $\chi(x)=\omega^{t_{\chi}\cdot x}$ where $\omega=e^{2\pi i/p}$. 
The Fourier transform of a function $f:G\to \mathbb{C}$ is defined as 
\[
\widehat{f}(\chi)=\frac{1}{|G|}\sum_{x\in G}f(x)\chi(x).
\]

A subspace $H$ is {\it $\delta$-weakly-regular
with respect to $f$} if for every character $\chi$, we have $|\widehat{f-f_H}(\chi)|\leq\delta$.



The weak regularity lemma claims that, for any function $f:\mathbb{F}_p^n \rightarrow [0,1]$ there is a subspace of bounded codimension that is weakly regular with respect to $f$. 
\begin{lem}
\label{lem:(Weak-regularity-lemma.)}(Weak regularity lemma.) For
any function $f:\mathbb{F}_{p}^{n}\rightarrow[0,1]$, there is a subspace
$H$ which is $\delta$-weakly-regular with respect to $f$ such that
$H$ has codimension $\left\lfloor\delta^{-2}\right\rfloor$. 
\end{lem}

The proof of Lemma \ref{lem:(Weak-regularity-lemma.)} is essentially that of Proposition 3.5 in \cite{Green06}, referred to there as the linear Koopman von-Neumann decomposition. We include the proof here for convenience.

\begin{proof}

Consider the set $S$ of nonzero characters $\chi$ where $|\widehat{f}(\chi)| \ge \delta$. By Parseval's identity, $1\ge \frac{1}{p^n} \sum_{x\in \mathbb{F}_p^n} f(x)^2 = \sum_{\chi \in \widehat{\mathbb{F}_p^n}} |\widehat{f}(\chi)|^2 \ge |S| \delta^2$. Hence $|S|\le \left\lfloor\delta^{-2}\right\rfloor$.

Consider any subspace $H \subseteq \{x\in \mathbb{F}_p^n:~ \chi(x)=1~\forall \chi \in S\}$ of codimension $\left\lfloor\delta^{-2}\right\rfloor$, which exists since $|S|\le \left\lfloor \delta^{-2}\right\rfloor$ and $\{x\in \mathbb{F}_p^n:~ \chi(x)=1~\forall \chi \in S\} = \{x\in \mathbb{F}_p^n:~ x \cdot t_\chi=0~\forall \chi \in S\}$ has codimension at most $|S|$. Let $\mu_H(x)=\frac{I_H(x)}{|H|}$ where $I_H(x)=1$ if $x\in H$ and $I_H(x)=0$ otherwise. We have $\widehat{\mu_H}(\chi)=\frac{1}{p^n}$ if $t_\chi\in H^{\perp}$, as then $\chi(x)=1$ for all $x$ such that $\mu_H(x)\ne 0$, and $\widehat{\mu_H}(\chi)=0$ otherwise. Observe that 
\[
f_H(x) = \frac{\sum_{y\in H+x}f(y)}{|H|} = \sum_{z} \mu_H(z)f(x-z)=p^n \cdot \mu_H * f(x).
\] Thus, $\widehat{f_H}(\chi) = p^n \widehat{\mu_H}(\chi)\widehat{f}(\chi)$ and hence 
\[
\widehat{f-f_H}(\chi) = \widehat{f}(\chi)\left(1-p^n\widehat{\mu_H}(\chi)\right).
\] Since $1-p^n\widehat{\mu_H}(\chi)=0$ for $\chi \in S$ as $t_\chi \in H^{\perp}$ for all $\chi \in S$, $\widehat{f-f_H}(\chi)=0$ for $\chi\in S$. For $\chi\notin S$, $|\widehat{f}(\chi)|\le \delta$, hence $$|\widehat{f-f_H}(\chi)| = |\widehat{f}(\chi)\left(1-p^n\widehat{\mu_H}(\chi)\right)| \le |\widehat{f}(\chi)| \le \delta,$$ where the first inequality follows since $p^n\widehat{\mu_H}(\chi)$ is either $0$ or $1$. 
\end{proof}

We say two function $f,g:\mathbb{F}_{p}^{n}\rightarrow[0,1]$ are
\textit{$\delta$-close} if for every character $\chi$ we have $|\widehat{f-g}(\chi)|\leq\delta$.
Thus, $H$ is $\delta$-weakly-regular with respect to $f$ if and
only if $f$ and $f_{H}$ are $\delta$-close. The $3$-AP density
$\Lambda(f)$ of $f$ is defined as: 
\[
\Lambda(f)=\mathbb{E}_{x-2y+z=0}\left[f(x)f(y)f(z)\right].
\]

This expectation is over all triples $x,y,z\in G$ with $x-2y+z=0$.
Note that this includes the trivial arithmetic progressions $x=y=z$
with common difference $0$. 
More generally, for functions $f_1,f_2,f_3$, we define 
\[
\Lambda(f_1,f_2,f_3) = \mathbb{E}_{x-2y+z=0}\left[f_1(x)f_2(y)f_3(z)\right]
\]
Note that $\Lambda(f)=\Lambda(f,f,f)$. 
The counting lemma shows that the
$3$-AP density of a function can be approximated by that of any function
close to it.

\begin{lem}\label{lem:Counting lemma}(Counting lemma.) If $f,g:\mathbb{F}_{p}^{n}\rightarrow[0,1]$
are $\delta$-close and $f$ has density $\alpha$, then
$|\Lambda(f)-\Lambda(g)|\leq3\delta\alpha$.\end{lem}

The above bound shows that if functions $f$ and $g$ are $\delta$-close and $f$ has density $\alpha$, then the $3$-AP densities
of $f$ and $g$ are at most $3\delta\alpha$ apart. In particular, if $H$ is $\delta$-weakly-regular with respect to $f$ then we can approximate the $3$-AP density of $f$ by that of $f_H$. 

\begin{proof}

We have 
\[
\Lambda(f)-\Lambda(g) = \Lambda(f-g,f,f)+\Lambda(g,f-g,f)+\Lambda(g,g,f-g)
\]

The following identity (see, e.g., Lemma 1.7 in \cite{Green06}) is key to the original proof of Roth's theorem:  
\[
\Lambda(h_1,h_2,h_3)=\sum_{\chi} \widehat{h_1}(\chi)\widehat{h_2}(-2\chi)\widehat{h_3}(\chi). 
\] It follows that 
\begin{eqnarray}
|\Lambda(f-g,f,f)| &=& |\sum_{\chi} \widehat{f-g}(\chi)\widehat{f}(-2\chi)\widehat{f}(\chi)|\\
& \le & \sup_{\chi} |\widehat{f-g}(\chi)| (\sum_{\chi} \widehat{f}(-2\chi)^2)^{1/2}(\sum_{\chi} \widehat{f}(\chi)^2)^{1/2} \\
& \le & \delta \cdot \left(\frac{1}{p^n} \sum_{x\in \mathbb{F}_p^n} f(x)^2\right) \le \delta \alpha,
\end{eqnarray} where the first inequality is by the Cauchy-Schwarz inequality, and the last inequality follows from the fact that $f$ takes values in $[0,1]$ with expected value $\alpha$. Similarly, $|\Lambda(g,f-g,f)| \le \delta \alpha$ and $|\Lambda(g,g,f-g)|\le \delta \alpha$. Hence,
$|\Lambda(f)-\Lambda(g)| \le 3\delta \alpha$, which completes the proof. \end{proof}

\section{Upper bound}

\label{sectionupperbound}

The following theorem is the main result in this section and gives the upper bound in Theorem \ref{main1}. The upper bound in this theorem applies more generally to weighted set in $\mathbb{F}_p^n$, which is given by a function $f:\mathbb{F}_p^n\rightarrow [0,1]$. We define the density of $3$-APs with common difference $d$ of a weighted set $f:\mathbb{F}_{p}^{n}\to [0,1]$ as $\mathbb{E}_{x\in\mathbb{F}_p^n}[f(x)f(x+d)f(x+2d)]=\frac{1}{p^n}{\sum_{x\in\mathbb{F}_p^n}[f(x)f(x+d)f(x+2d)]}$. The density of $3$-APs with common difference $d$
of a set $A$ is the same as that of the characteristic function
of $A$.

\begin{thm}
\label{thm:largebound} If $n$ is at least an exponential tower of $p$'s of height $\log((\alpha-\alpha^3)/\epsilon)+5$ with a $1/\epsilon$ on top, then for any function $f:\mathbb{F}_p^n \rightarrow [0,1]$ of density $\alpha$, there is a nonzero $d$ in $\mathbb{F}_p^n$ such that the density of $3$-APs with common difference $d$ is at least $\alpha^3-\epsilon$. 
\end{thm}




A $3$-AP with common difference $d$ is an ordered triple $(a,b,c)$
such that $c-b=b-a=d$. A $3$-AP is {\it trivial} if the common difference $d$ is zero, i.e., it contains
the same element three times. Otherwise, we call the $3$-AP {\it nontrivial}. 

Let $G=\mathbb{F}_{p}^{n}$. For each affine subspace $H$ of $\mathbb{F}_{p}^{n}$,
let $\alpha(H)=\mathbb{E}_{x\in H}[f(x)]$ be the density of $f$
in $H$. Then $\alpha(G)=\mathbb{E}_{x\in G}[f(x)]=\alpha$ is the
density of $f$. The \textit{mean cube density} $b(H)$ is defined
to be the average of the cube of the density of $f$ in the affine
translates of $H$ which partition $\mathbb{F}_{p}^{n}$. It is also
given by $b(H)=\mathbb{\mathbb{E}}_{g\in G}[\alpha(H+g)^{3}]$, where
$H+g=\{h+g:h\in H\}$ is the affine translate of $H$ by $g$.

Recall that $\Lambda(f)$ denotes the $3$-AP density of a function $f:\mathbb{F}_{p}^{n}\rightarrow[0,1]$. If the
function $f$ is well understood from context, then, for an affine
subspace $H$, we let $\Lambda(H)$ denote the density of three-term
arithmetic progressions of $f$ in $H$. That is, 
\[
\Lambda(H)=\mathbb{E}_{x,y,z\in H,~x-2y+z=0~}[f(x)f(y)f(z)].
\]
We let $\lambda(H)$ denote the density of \emph{nontrivial} three-term
arithmetic progressions of $f$ in $H$. That is, 
\[
\lambda(H)=\mathbb{E}_{x,y,z\in H~\textrm{distinct},~x-2y+z=0~}[f(x)f(y)f(z)].
\]
Observe that $\lambda(H)$ and $\Lambda(H)$ are close if $H$ is
large. Indeed, 
\begin{equation}
\lambda(H)=\frac{\Lambda(H)\cdot|H|^{2}-|H|\cdot\mathbb{E}_{x\in H}\left[f(x)^{3}\right]}{|H|(|H|-1)}\ge\Lambda(H)-\frac{\mathbb{E}_{x\in H}\left[f(x)^{3}\right]}{|H|},\label{closeLl}
\end{equation}
where the equality follows from the fact that there are $|H|(|H|-1)$
nontrivial $3$-AP in the affine subspace $H$ as each $3$-AP is determined
by its first two elements. By averaging the previous inequality over
all translates of $H$ and letting $\alpha$ denote the average value
of $f$, we have 
\begin{equation}
\mathbb{E}_{g}[\lambda(H+g)]\ge\mathbb{E}_{g}[\Lambda(H+g)]-\frac{\mathbb{E}_{g\in G}\left[\mathbb{E}_{x\in H+g}\left[f(x)^{3}\right]\right]}{|H|}\ge\mathbb{E}_{g}[\Lambda(H+g)]-\frac{\alpha}{|H|}.\label{closeLlL}
\end{equation}

The proof of Theorem \ref{thm:largebound} is by a density increment
argument using the mean cube density. It is common in regularity lemmas
and related results to use a density increment argument using the
mean square density. 
However, there have been
a couple of instances already where a different, carefully chosen
density function is used in order to make the proof work. In the first
author's improved bound on the graph removal lemma \cite{F11}, the
mean entropy density was used, where the entropy function is $f(x)=x\log x$.
Another instance is in Scott's proof of the sparse graph regularity
lemma \cite{Scott}, which uses a function which is first $f(x)=x^{2}$,
but then becomes linear from some point onwards.

The next lemma shows that if the density of $3$-APs with nonzero common
difference in a subspace $H$ is small, then the mean cube density
can be increased substantially by passing to a subspace $H'$ of bounded
codimension. In fact, the difference between the mean cube density
and the cube of the total density increases by more than a factor two. 
\begin{lem}
\label{lem:mean cubed density increment_small}
If $f:\mathbb{F}_{p}^{n}\rightarrow[0,1]$ has density $\alpha$, $H$ is a subspace of $\mathbb{F}_{p}^{n}$ of size larger than $4\alpha/\epsilon$, and the density of $3$-APs with nonzero common difference in $H$ of $f$ is less than $\alpha^{3}-\epsilon$, then there is a subspace
$H'$ of $H$ with $\text{Codim}(H')\le\text{Codim}(H)+p^{\text{Codim}(H)}\cdot144/\epsilon^{2}$
such that $b(H')-\alpha^{3} > 2(b(H)-\alpha^{3})+\epsilon/2$.\end{lem}
\begin{proof}
Denote the translates of $H$ by $H_{j}$ for $j\in \mathbb{F}_p^n/H$, so each affine translate of $H$ is labeled by the corresponding element in $\mathbb{F}_p^n/H$.
Let $\eta=\epsilon/12$. For each affine translate $H_{j}$ of $H$,
we apply the weak regularity lemma, Lemma \ref{lem:(Weak-regularity-lemma.)},
within $H_{j}$ to obtain an $\eta$-weakly-regular subspace $T_{j}$
containing $0$ with codimension $M=\left\lfloor\eta^{-2}\right\rfloor$
in $H$. Consider the average function $t_{j}:H_{j}\to[0,1]$, which 
is constant on each affine translate of $T_{j}$ in $H_{j}$ and whose
value is the average of $f$ on this affine translate. By assumption,
the density of $3$-APs with nonzero common differences 
in $H$ is 
\begin{equation}
\mathbb{E}_{j}[\lambda(H_{j})]<\alpha^3-\epsilon.\label{assump}
\end{equation}
By the counting lemma, Lemma \ref{lem:Counting lemma}, we have 
\[
|\Lambda(H_{j})-\Lambda(t_{j})|\leq3\alpha(H_{j})\eta\le3\eta,
\] where $\Lambda(t_j)$ is the density of $3$-APs in $H_j$ with the weights given by the function $t_j$. 
Hence, 
\begin{equation}
\mathbb{E}_{j}[\Lambda(H_{j})]\geq\mathbb{E}_{j}[\Lambda(t_{j})]-3\eta.\label{eq:counting-inequality}
\end{equation}

Denote the affine translates of $T_{j}$ in $H_{j}$ by $T_{j,k}$,
 for $k\in H/T_j$. Also, let
$C$ be the set of all ordered triples $(k_{1},k_{2},k_{3})$ such
that $k_1,k_2,k_3$ form a $3$-AP in $H/T_j$
(including the ones where $k_1=k_2=k_3$), which is equivalent to the fact that $T_{j,k_1},T_{j,k_2},T_{j,k_3}$ form a $3$-AP of affine translates of $T_j$ in $H_j$. We have that each $(k_{1},k_{2})$ appears
in $C$ exactly once, so $|C|=p^{2M}$. Moreover, we have 
\begin{eqnarray}
\Lambda(t_{j}) & = & \mathbb{E}_{(k_{1},k_{2},k_{3})\in C}\left[t_{j}(T_{j,k_{1}})t_{j}(T_{j,k_{2}})t_{j}(T_{j,k_{3}})\right],\label{eq:121a}
\end{eqnarray}
where $t_j(T_{j,k})$ is the constant value $t_j(x)$ for $x \in T_{j,k}$. 

Hence, 
\begin{eqnarray}
\mathbb{E}_{j}\left[\lambda(H_{j})\right] & \ge & \mathbb{E}_{j}\left[\Lambda(H_{j})\right]-\frac{\alpha}{|H|}\nonumber \\
 & \ge & \mathbb{E}_{j}\left[\Lambda(t_{j})\right]-3\eta-\frac{\alpha}{|H|}\nonumber \\
 & = & \mathbb{E}_{j} \mathbb{E}_{(k_{1},k_{2},k_{3})\in C}\left[t_{j}(T_{j,k_{1}})t_{j}(T_{j,k_{2}})t_{j}(T_{j,k_{3}})\right]-3\eta-\frac{\alpha}{|H|}\nonumber \\
 & > & \mathbb{E}_{j} \mathbb{E}_{(k_{1},k_{2},k_{3})\in C}\left[t_{j}(T_{j,k_{1}})t_{j}(T_{j,k_{2}})t_{j}(T_{j,k_{3}})\right]-6\eta,\label{last1}
\end{eqnarray}
where the first inequality is from (\ref{closeLlL}), the second inequality
is from (\ref{eq:counting-inequality}), the equality is by (\ref{eq:121a}),
and the last inequality follows from the condition that $|H|>4\alpha/\epsilon$,
so $\frac{\alpha}{|H|}<\epsilon/4=3\eta$.

Schur's inequality says that for any nonnegative real numbers $a,b,c$, we have
$$a^{3}+b^{3}+c^{3}+3abc\ge a^{2}(b+c)+b^{2}(a+c)+c^{2}(a+b).$$ In the following sequence of inequalities, we fix $j$ and take the expectations over $(k_1,k_2,k_3)\in C$. Applying 
Schur's inequality to $t_{j}(T_{j,k_{1}}),t_{j}(T_{j,k_{2}}),t_{j}(T_{j,k_{3}})$ for
$(k_{1},k_{2},k_{3})\in C$, we have 
\begin{align} \label{eq:schur_cor}
\nonumber & \mathbb{E}\left[t_{j}(T_{j,k_{1}})^{3}+t_{j}(T_{j,k_{2}})^{3}+t_{j}(T_{j,k_{3}})^{3}+3t_{j}(T_{j,k_{1}})t_{j}(T_{j,k_{2}})t_{j}(T_{j,k_{3}})\right]\\
\nonumber & \ge\mathbb{E}\left[t_{j}(T_{j,k_{1}})^{2}(t_{j}(T_{j,k_{2}})+t_{j}(T_{j,k_{3}}))+t_{j}(T_{j,k_{2}})^{2}(t_{j}(T_{j,k_{1}})+t_{j}(T_{j,k_{3}}))+t_{j}(T_{j,k_{3}})^{2}(t_{j}(T_{j,k_{1}})+t_{j}(T_{j,k_{2}}))\right]\\
\nonumber & =\frac{1}{p^{2M}}\left(6\sum_{k_{1},k_{2}\in H/T_j}t_{j}(T_{j,k_{1}})^{2}t_{j}(T_{j,k_{2}})\right)\\
 & =\frac{6}{p^{2M}}\left(\sum_{k\in H/T_j}t_{j}(T_{j,k})^{2}\right)\left(\sum_{k\in H/T_j}t_{j}(T_{j,k})\right)\ge6\alpha(H_{j})^{2}\alpha(H_{j})=6\alpha(H_{j})^{3},
\end{align}
where the first equality comes from the fact that fixing any element
$k_{1}\in H/T_j$ and $k_{2}\in H/T_j$ and two different positions (first, second or
third) of $k_1,k_2$ in a $3$-AP of subspaces, we can find a unique $3$-AP of subspaces with two
positions specified, and the last inequality comes from the Cauchy-Schwarz
inequality, noting that $\mathbb{E}_{k\in H/T_j}[t_{j}(T_{j,k})]=\alpha(H_{j})$.


Taking $H'=H\cap\left(\bigcap_{j}T_{j}\right)$,
we have 
\begin{align*}
b(H') & \ge\mathbb{E}_{j,k}\left[t_{j}(T_{j,k})^{3}\right]\\
 & =\mathbb{E}_{j}\left[\frac{1}{3}\mathbb{E}_{(k_{1},k_{2},k_{3})\in C}\left[t_{j}(T_{j,k_{1}})^{3}+t_{j}(T_{j,k_{2}})^{3}+t_{j}(T_{j,k_{3}})^{3}\right]\right]\\
 & \ge\mathbb{E}_{j}\left[2\alpha(H_{j})^{3}-\mathbb{E}_{(k_{1},k_{2},k_{3})\in C}\left[t_{j}(T_{j,k_{1}})t_{j}(T_{j,k_{2}})t_{j}(T_{j,k_{3}})\right]\right]\\
 & > 2b(H)-\mathbb{E}_{j}[\lambda(H_{j})]-6\eta\\
 & > 2b(H)-(\alpha^{3}-\epsilon)-6\eta,
\end{align*}
where the first inequality follows from Jensen's inequality applied
to the convex function $h(x)=x^{3}$, noting that the partition by
$H'$ is a refinement of the partition by translates of $T_{j}$ in each affine
subspace $H_{j}$, the second inequality follows from (\ref{eq:schur_cor}),
the third inequality is by (\ref{last1}), and the last inequality comes
from the assumption that $\mathbb{E}_j[\lambda(H_{j})]<\alpha^{3}-\epsilon$. It follows that 
\[
b(H')-\alpha^{3}\ge2(b(H)-\alpha^{3})+\epsilon/2,
\]
where we used $\epsilon-6\eta=\epsilon/2$.
Finally, we bound the codimension of the subspace $H'$: 
\begin{eqnarray*} \text{Codim}(H') & = & \text{Codim}\left(H\cap\left(\bigcap_{j}T_{j}\right)\right) \le\text{Codim}(H)+\eta^{-2}\cdot p^{\text{Codim}(H)} \\ & = & \text{Codim}(H)+p^{\text{Codim}(H)}\cdot144/\epsilon^{2}.
\end{eqnarray*}
Thus the subspace $H'$ has the desired properties. 
\end{proof}

The proof of Theorem \ref{thm:largebound} follows from repeatedly applying Lemma \ref{lem:mean cubed density increment_small}. 

\begin{proof}[Proof of Theorem \ref{thm:largebound}.]

Let $f:\mathbb{F}_p^n\rightarrow [0,1]$ be such that the density of $3$-APs with any fixed nonzero common difference
of $f$ is less than $\alpha^{3}-\epsilon$. Let $H_{0}=\mathbb{F}_{p}^{n}$,
so $b(H_{0})=\alpha^{3}$. We define a sequence of subspaces $H_{0}\supset H_{1} \supset \cdots \supset H_s$ recursively as follows. Note that this implies 
$b(H_0) \leq b(H_1) \leq \ldots \leq b(H_s)$. If $|H_i| \geq 4\alpha/\epsilon$, then we apply Lemma \ref{lem:mean cubed density increment_small} to obtain a subspace 
$H_{i+1} \subset H_i$ with $$b(H_{i+1})-\alpha^3 \geq 2(b(H_i)-\alpha^3)+\epsilon/2$$ and 
\[
\text{Codim}(H_{i+1})\le\text{Codim}(H_{i})+p^{\text{Codim}(H_{i})}\cdot144/\epsilon^{2}.
\]
It follows that $2\text{Codim}(H_{i+1}) \leq \max\left(145^2\epsilon^{-4},p^{2\text{Codim}(H_{i})}\right)$. In particular, 
$\text{Codim}(H_{i+1})$ is at most a tower of $p$'s of height $i$ with a $145^2\epsilon^{-4}$ on top. Observe that $145^2\epsilon^{-4}<p^{p^{p^{1/\epsilon}}}$, so 
$\text{Codim}(H_{i+1})$ is at most a tower of $p$'s of height $i+3$ with a $1/\epsilon$ on top.  By induction on $i$, we have 
\begin{equation} b(H_{i}) \geq \alpha^3 +(2^i-1)\epsilon/2. \nonumber
\end{equation}

By convexity of $h(x)=x^3$, we have $b(H) \leq \alpha$ for every subspace $H$. Hence, $\alpha \geq b(H_i) > \alpha^3 +(2^i-1)\epsilon/2$ for each $i$, from which we conclude that $s \leq 2+\log\left((\alpha-\alpha^3)/\epsilon \right)$. Hence, $\text{Codim}(H_s)$ is at most a tower of $p$'s of height $s+2 \leq 4+\log\left((\alpha-\alpha^3)/\epsilon \right)$ with a $1/\epsilon$ on top. We must have 
$|H_s|<4 \alpha / \epsilon$ in order to not be able to apply Lemma \ref{lem:mean cubed density increment_small} and obtain the next subspace in the sequence. As $p^n=p^{\text{Codim}(H_s)}|H_s|$, we have $n<\text{Codim}(H_s)+\log_p(4\alpha/\epsilon)$, which completes the proof.
\end{proof}
 
\section{Lower bound}

\label{sectionlowerbound}

The following theorem is the main result in this section and gives the lower bound in Theorem \ref{main1}.

\begin{thm}
\label{thm:Lower bound for line density} For $0<\alpha\leq1/2$ and
$\epsilon\leq 2^{-161}p^{-8}\alpha^3$, there
exists $A \subset \mathbb{F}_{p}^{n}$ of density at least $\alpha$, where $n$ is a tower
of $p$'s of height at least $\frac{1}{52}\log(\alpha^{3}/\epsilon)$,
such that for all nonzero $d$ in $\mathbb{F}_{p}^{n}$, the density
of $3$-APs with common difference $d$ of $A$ is less than $\alpha^{3}-\epsilon$. 
\end{thm}

Our goal for the remainder of the paper is to prove Theorem \ref{thm:Lower bound for line density}.

\subsection{From weighted to unweighted}

For the construction of the set $A$ in Theorem \ref{thm:Lower bound for line density},
it will be more convenient to work with a weighted set in $\mathbb{F}_{p}^{n}$,
which is given by a function $f:\mathbb{F}_{p}^{n}\to[0,1]$. The
weighted analogue of Theorem \ref{thm:Lower bound for line density}
is given below.
Note that for the weighted construction, it will be convenient to normalize and 
replace $\epsilon$ by $\epsilon\alpha^{3}$. 

\begin{thm}
\label{thm:Lower bound for weighted set} Let $0 < \alpha \leq 1/2$, $p$ be an odd prime, and $\epsilon \leq 2^{-160}p^{-8}$. 
There exists a function $f:\mathbb{F}_{p}^{n}\to[0,1]$ of density $\alpha$, where $n$
is a tower of $p$'s of height at least $\frac{1}{52}\log(2/\epsilon)$, such that for each nonzero $d \in \mathbb{F}_p^n$, the
density of $3$-APs with common difference $d$ of $f$ is less than $(1-\epsilon)\alpha^{3}$. 
\end{thm}

First, we prove that Theorem \ref{thm:Lower bound for line density}
follows from Theorem \ref{thm:Lower bound for weighted set}.
We do this by considering a random set, with each element $x$ in
$\mathbb{F}_{p}^{n}$ included with probability $f(x)$ independently
of the other elements. An application of Hoeffding's inequality allows
us to show that it is unlikely that the density of the random set
deviates much from that of the weighted set. Further, for each common
difference, it is very unlikely that the density of $3$-AP with that
common difference deviates much from the density in the weighted set.
A union bound then shows that the random set likely has density and
density of $3$-APs with each common difference close to that of the weighted
set. 
\begin{lem}
\label{lem:density to weighted} If $n$ is a postive integer, $p$ a prime number, $f:\mathbb{F}_{p}^{n}\to[0,1]$,
$N=p^{n}$, and $\epsilon\geq 2\left(\frac{\ln (12N)}{N}\right)^{1/2}$, then there exists a subset $A\subset\mathbb{F}_{p}^{n}$ such that
the density of $A$ and, for each nonzero $d \in \mathbb{F}_p^n$, the density of $3$-APs with common difference $d$ of $A$ deviate no more than $\epsilon$ from those of $f$. \end{lem}
\begin{proof}
Consider a random set $A$ with each $x\in\mathbb{F}_{p}^{n}$ having
probability $f(x)$ of being in $A$, independently of the other elements.

The number of $3$-APs with common difference $d$ in $\mathbb{F}_p^n$ is $p^n$, determined precisely by the first element of the $3$-AP. However, the number of distinct $3$-APs with common difference $d$ (distinct in that the set of three terms is distinct) is $p^n$ for $p>3$ and $p^{n-1}$ for $p=3$. For $p=3$, the $p^{n-1}$ sets of distinct $3$-APs are disjoint. However, for $p>3$, some pairs of distinct $3$-APs have nonempty intersection. 
 
For $p>3$, we partition the $3$-APs with the same common difference $d$ in $A$ into
different classes such that any two $3$-APs in one class are disjoint.
Given any element $x$ in $\mathbb{F}_{p}^{n}$ and $i\in\{1,2,3\}$
such that $x$ is the $i$th element of the $3$-AP with common difference
$d$, this $3$-AP is uniquely determined. Hence, for any $3$-AP in $\mathbb{F}_{p}^{n}$
with common difference $d$, there are exactly four other $3$-APs with common difference $d$ having nonempty
intersection with it (the first element of the $3$-AP is the second
or third element of two other such $3$-APs, and the second or the third
element of the $3$-AP is the first element of two other such $3$-APs). We can define an auxiliary graph with vertex set $V$ consisting of the $3$-APs with common difference $d$, and two $3$-APs are adjacent if they have nonempty intersection. This auxiliary graph thus has maximum degree at most four. As every graph with maximum degree at most four has chromatic number at most five, we can partition the set of $3$-APs with common difference $d$ into five different
classes where any pair of $3$-APs in one class are disjoint. Further, by the Hajnal-Szemer\'edi theorem \cite{HS}, there is such a partition into five parts which is equitable, so that each class has size $\left\lfloor |V|/5 \right\rfloor$ or $\lceil |V|/5 \rceil$. 

Each class has size at least $\left\lfloor |V| /5 \right\rfloor = \left\lfloor p^n /5 \right\rfloor \geq p^n/7$. Recall that $N=p^n$. By Hoeffding's inequality, for each such class of $3$-APs, the probability that the density of $3$-APs in that class of $A$ differs from the density in $f$ by more than $\epsilon$ is less than $2\exp(-2\epsilon^{2}N/7)$. 
Thus, by the union bound, the probability that the density of $A$ and $f$ differ by more than $\epsilon$ in at least 
one of the five classes of disjoint $3$-APs with common difference $d$ is less than $10\exp(-2\epsilon^{2}N/7)$.
Also by Hoeffding's inequality, the probability that the usual density of
$A$ differs from that of $f$ by more than $\epsilon$ is at most $2\exp(-\epsilon^{2}N)$.
Hence, by the union bound, the probability that $A$ satisfies that the density of $A$ or the density of $3$-APs with common difference $d$ of $A$ for some nonzero $d$ deviates from those of $f$ by more than $\epsilon$ is at most 
\[
(N-1)\cdot10\exp(-2\epsilon^{2}N/7)+2\exp(-\epsilon^{2}N)<12N\exp(-2\epsilon^{2}N/7)<1,
\]
as $\epsilon^{2}\geq 4 \frac{\ln (12N)}{N}$ . Therefore, there exists a set
$A$ such that the density of $A$ and of $f$ differ by at most $\epsilon$, and, for each nonzero $d$, the density of $3$-APs with common difference $d$ in $A$ and in $f$ differ by at most $\epsilon$. 
\end{proof}

We next show how to obtain Theorem \ref{thm:Lower bound for line density}
from Theorem \ref{thm:Lower bound for weighted set}. 

\noindent {\bf Proof of Theorem \ref{thm:Lower bound for line density}:} Let $0<\alpha\leq1/2$
and $\epsilon \leq \frac{1}{2}(2^{20}p)^{-8}\alpha^3$. Let
$n$ and $f$ satisfying the conclusion of Theorem \ref{thm:Lower bound for weighted set}
for $\alpha$ and $\epsilon'=2\epsilon/\alpha^3$. In particular, $n$ is at least a tower of $p$'s of height $\frac{1}{52}\log (2/\epsilon')=\frac{1}{52}\log(\alpha^3/\epsilon)$. Apply Lemma \ref{lem:density to weighted} with $\epsilon^{*}=\epsilon/4$
and this $n$ and $f$ to obtain a set $A$ satisfying the conclusion
of Lemma \ref{lem:density to weighted}. By the lower bound on $n$, we have 
\[
2\left(\frac{\ln (12p^{n})}{p^{n}}\right)^{1/2}<p^{-n/3}<\epsilon/4.
\]
We obtain a set whose density is in $[\alpha-\epsilon/4,\alpha+\epsilon/4]$
and such that the density of $3$-APs for each nonzero common difference is less than $(1-\epsilon')\alpha^{3}+\epsilon/4=\alpha^{3}-7\epsilon/4$.
Now, we simply delete or add arbitrary elements to make the set have density $\alpha$. The $3$-AP density for each common difference
can change by at most by $3\epsilon/4$, so the density of $3$-APs for each nonzero common difference in the set is less than $\alpha^{3}-7\epsilon/4+3\epsilon/4=\alpha^{3}-\epsilon$. \qed

\vspace{0.1cm}
\noindent {\bf Construction idea} 
\vspace{0.1cm}

In the next subsection, we prove Theorem \ref{thm:Lower bound for weighted set}. The general idea is as follows. We partition the dimension $n=m_1+m_2+\cdots+m_s$, where $m_{i+1}$ is roughly exponential in $m_i$ for each $i$, and let $n_i=m_1+m_2+\cdots+m_i$ be the $i^{\textrm{th}}$ partial sum, so $n_1=m_1$ and $n_{i}=n_{i-1}+m_i$ for $2 \leq i \leq s$. Consider the vector space as a product of smaller vector spaces: $\mathbb{F}_p^{n} = \mathbb{F}_p^{m_1} \times \mathbb{F}_p^{m_2} \times \cdots \times \mathbb{F}_p^{m_s}$. In each step $i$, we determine a partial function $f_i:\mathbb{F}_p^{n_i} \rightarrow [0,1]$ with density $\alpha$. The function $f_i$ has the property that for each nonzero $d \in \mathbb{F}_p^{n_i}$, the density of $3$-APs with common difference $d$ of $f_i$ is less than  $(1-\epsilon)\alpha^3$. 

We need a starting point, picking $m_1$ and $f_1$ appropriately. We pick $m_1=\left\lfloor \frac{1}{2}\log_p(3/\epsilon)\right\rfloor$ and $f_1(0)=(1-(p^{m_1}-1)\eta)\alpha$ with $\eta=\sqrt{\epsilon/3}$, and let $f_1(x)=(1+\eta)\alpha$ for $x \not = 0$ so that the average value is $\alpha$. We easily verify that $f_1$ has the desired properties. 

For $i \geq 2$, observe that we can use $f_{i-1}$ to define a function $g_i:\mathbb{F}_p^{n_{i}} \rightarrow [0,1] $ by letting $g_i(x)=f_{i-1}(y)$, where $y$ is the first $n_{i-1}$ coordinates of $x$. Thus, $g_i$ has constant value $f_{i-1}(y)$ on the copy of $\mathbb{F}_p^{m_i}$ consisting of those elements of $\mathbb{F}_p^{n_{i}}$ whose first $n_{i-1}$ coordinates are $y$. We perturb $g_i$ to obtain $f_{i}$  so that it has several useful properties. 

Before explaining how this perturbation is done exactly, we first describe some of the useful properties $f_{i}$ will have.  While $g_i$ has constant value $f_{i-1}(y)$ on each copy of $\mathbb{F}_p^{m_i}$ whose first $n_{i-1}$ coordinates is $y$, the function $f_{i}$ will not have this property, but will still have average value $f_{i-1}(y)$ on each of these copies. Another useful property is that for each $d \in \mathbb{F}_p^{n_i}$ such that $d$ is not identically $0$ on the first $n_{i-1}$ coordinates, the density of $3$-APs with common difference $d$ in $f_{i}$ is equal to the density of $3$-APs with common difference $d^{*}$ in $f_{i-1}$, where $d^{*} \in \mathbb{F}_p^{n_{i-1}} \setminus \{0\}$ is the first $n_{i-1}$ coordinates of $d$. Once we have established this property, it suffices then to check that for each nonzero $d \in \mathbb{F}_p^{n_i}$ with the first $n_{i-1}$ coordinates of $d$ equal to $0$, the density of $3$-APs with common difference $d$ is less than $(1-\epsilon) \alpha^3$. In order to check this, it now makes sense to explain a little more about how we obtain $f_i$ from $g_i$. 

Consider a set $B \subset \mathbb{F}_p^{m_{i}}$ with relatively few three-term arithmetic progressions (considerably less than the random bound) given its size. We take $B$ to be the elements whose first coordinate is in an interval of length roughly $2p/3$ in $\mathbb{F}_p$. 

We let $\mathcal{C}$ be an appropriately chosen subcollection of the $p^{n_{i-1}}$ copies of $\mathbb{F}_p^{m_{i}}$ in $\mathbb{F}_p^{n_{i}}$, where each copy has the first $n_i$ coordinates fixed to some $y \in \mathbb{F}_p^{n_{i-1}}$. If $x \in \mathbb{F}_p^{n_{i}}$ is in a copy of $\mathbb{F}_p^{m_{i}}$ not in $\mathcal{C}$, then we let $f_{i}(x)=g_i(x)$. In other words, we leave $f_{i}$ constant on the affine subspaces not in $\mathcal{C}$. For each $A \in \mathcal{C}$, we consider a random copy of $B$ in $A$ by taken a random linear transformation of full rank from $\mathbb{F}_p^{m_{i}}$ to $A$ and consider the image of $B$ by this linear transformation, and then scale the weights by the constant factor $p^{m_i}/|B|$ to keep the average weight unchanged on $A$. We do this independently for each $A \in \mathcal{C}$. We show that with high probability, for every nonzero $d \in \mathbb{F}_p^{n_{i}}$ with the first $n_{i-1}$ coordinates of $d$ equal to $0$, the density of $3$-APs with common difference $d$ is less than $(1-\epsilon) \alpha^3$. One can show this for each such $d$ by observing that the density of $3$-APs with common difference $d$ is just the average of the densities of $3$-APs with common difference $d$ on each of the $p^{n_{i-1}}$ copies of $\mathbb{F}_p^{m_i}$. The density of $3$-APs with common difference $d$ in the subspaces not in $\mathcal{C}$ remain unchanged, but the densities of $3$-APs with common difference $d$ in the subspaces in $\mathcal{C}$ are independent random variables that have expected value (appropriately scaled) equal to the density of $3$-APs in $B$, which is much less than the random bound for a set of this size. We can then use Hoeffding's inequality, which allows us to show that the sum of a set of independent random variables with values in $[0,1]$ is highly concentrated on its mean, to show that it is very unlikely that the density of $3$-APs with common difference $d$ is at least $(1-\epsilon)\alpha^3$. Since the probability is so tiny, a simple union bound allows us to get this to hold simultaneously for all nonzero $d$. This completes the construction idea. 

Before proceeding with the detailed argument, in the next subsection we present a useful construction of a set which has relatively few three-term arithmetic progressions which serves as a main ingredient in our proof. 


\vspace{0.2cm}

\subsection{Subsets with relatively few arithmetic progressions} 
\label{relativelyfewsubs}
\vspace{0.2cm}

An important ingredient in our constructions is subsets of $\mathbb{F}_p$ with relatively few arithmetic progressions. We show that very large intervals in $\mathbb{F}_p$ have considerably fewer three-term arithmetic progressions than given by the random bound. In fact, Green and Sisask \cite{GrSi} proved more: if $\phi>0$ is small enough, then among all subsets of $\mathbb{F}_p$ of size $(1-\phi)p$, intervals have the fewest $3$-APs. 

\begin{lem}\label{largeinterval} 
Let $I \subset \mathbb{F}_p$ be an interval with $|I|=(1-\phi)p$ and $\phi \leq 1/2$. The density of $3$-APs in $I$ is at most $(1-\phi)^3-(\phi^2/2-\phi^3)$. 
\end{lem}
\begin{proof}
Let $J$ be the complement of $I$, which is an interval with $|J|=\phi p$. The total number of $3$-APs in $\mathbb{F}_p$ is $p^2$ as each $3$-AP is determined by its first element and common difference, and there are $p$ choices for each. We next count the number of three-term arithmetic progressions in $I$ (including the trivial ones with common difference $0$), which is $p^2$ minus the number of $3$-APs that intersect $J$. We count the number of $3$-APs that intersect $J$ using the inclusion-exclusion principle. For $i=1,2,3$, the number of $3$-APs with the $i^{\textrm{th}}$ term in $J$ is $\phi p^2$, as there are $|J|=\phi p$ choices for the $i^{\textrm{th}}$ term, and $p$ choices for the common difference. For $1 \leq i < j \leq 3$, the number of $3$-APs with the $i^{\textrm{th}}$ and $j^{\textrm{th}}$ in $J$ is $|J|^2=\phi^2p^2$ as the $3$-AP is determined by picking these two terms. The number of $3$-APs with all three elements in $J$ is just the number of $3$-APs in the interval $\{1,\ldots,|J|\}$ (possibly with negative common difference), which is $\lceil \frac{|J|^2}{2} \rceil=\lceil \frac{(\phi p)^2}{2} \rceil$. 
Thus, the number of $3$-APs in $I$  is 
\begin{eqnarray*}p^2-3\phi p^2+3\phi^2p^2-\left\lceil \frac{(\phi p)^2}{2} \right \rceil & \leq & p^2-3\phi p^2+3\phi^2p^2-\frac{(\phi p)^2}{2}\\ & = &  
\left((1-\phi)^3-\left(\frac{\phi^2}{2} -\phi^3\right)\right)p^2,\end{eqnarray*}
and the density of $3$-APs in $I$ is at most $(1-\phi)^3 -\left(\frac{\phi^2}{2} -\phi^3\right)$.
\end{proof}

We will use Lemma \ref{largeinterval} in our construction with $|I|=\lceil 2p/3 \rceil$ and so $\phi=1-\lceil 2p/3 \rceil/p$, which implies $\phi \approx 1/3$. 
Precisely $\phi=1/3$ if $p=3$, $\phi =\frac{1}{3}-\frac{1}{3p}$ if $p \equiv 1$ (mod $3$), and $\phi=\frac{1}{3}-\frac{2}{3p}$ if $p \equiv 2$ (mod $3$). It follows that $\frac{1}{5} \leq \phi \leq \frac{1}{3}$. 
Note that the expected value of the characteristic function of $I$ is $1-\phi$. Thus, the density of $3$-AP in $I$ is at least $\frac{\phi^2}{2} -\phi^3$ less than the random bound, which is $(1-\phi)^3$. For $b \in \mathbb{F}_p$, let $h(b)$ denote the density of $3$-APs in $I$ with common difference $b$. Thus, the average value of $h(b)$ is at least $\frac{\phi^2}{2} -\phi^3$ less than the random bound of $(1-\phi)^3$. Further, by the AM-GM inequality, it follows that the maximum value of $h$ is obtained when $b=0$, which is $h(0)= 1-\phi \leq 4/5$. It will be convenient to work with $\zeta=1-\phi=\lceil 2p/3 \rceil/p$ in the next subsection. 

\subsection{Proof of Theorem \ref{thm:Lower bound for weighted set}}

We first construct the function $f$, and then show that $f$
has the properties to verify Theorem \ref{thm:Lower bound for weighted set}.

\subsubsection{The construction}

Let $s=\left\lfloor \log_{90}(1/(8p\epsilon^{1/4}))\right\rfloor$. We next recursively define a sequence $m_{1},...,m_{s}$ of positive integers. We will let $n=\sum_{i=1}^{s}m_{i}$, $n_{j}=\sum_{i=1}^{j}m_{i}$ be the $j$th partial sum, and $N_i=|\mathbb{F}_p^{n_i}| = p^{n_i}$. Let $m_{1}=\left\lfloor \frac{1}{2}\log_{p}(3/\epsilon)\right\rfloor$. Let $\mu_{i}=90^{i}p\epsilon^{1/4}$ for $i \geq 1$, and $\sigma=10^4\ln p$.
For $i \geq 2$, let $m_{i}=\mu_{i}N_{i-1}/\sigma$. 

Before describing the construction and its properties, we first prove a couple of estimates which will be helpful later which show that $m_{i+1}$ is roughly exponential in $m_i$. First, we have \begin{eqnarray} \nonumber m_{2}& = & \mu_2N_1/\sigma > \mu_{2}\cdot(3/\epsilon)^{1/2}p^{-1}\cdot(10^4\ln p)^{-1} = 8100 p \epsilon^{1/4} \cdot (3/\epsilon)^{1/2}p^{-1}\cdot(10^4\ln p)^{-1}  >  \frac{ \epsilon^{-1/4}}{\ln p} \\ & > & \epsilon^{-1/8} \geq 2^{20}p, \label{m2ineq} \end{eqnarray} where we use in the last two
inequalities the bound $\epsilon \leq 2^{-160}p^{-8}$ given in the statement of Theorem \ref{thm:Lower bound for weighted set}. For $i\geq2$,
we have 
\begin{eqnarray} \nonumber m_{i+1} & = & \mu_{i+1}N_i/\sigma =90^{i+1}p\epsilon^{1/4} \cdot p^{n_i}/(10^4 \ln p) > \epsilon^{1/4}p^{n_i} = \epsilon^{1/4}p^{n_{i-1}}p^{m_i} \geq \epsilon^{1/4}p^{n_{1}}p^{m_i} \\ & > &  \epsilon^{1/4}\cdot (3/\epsilon)^{1/2}p^{-1} \cdot p^{m_i} > (\epsilon^{-1/4}p^{-1})  p^{m_i} \geq p^{m_i}.\label{miineq}
\end{eqnarray}

We divide the construction process into levels in order, starting
with level $1$ and ending at level $s$. In level $i$, we construct a function $f_i: \mathbb{F}_p^{n_i} \rightarrow [0,1]$ with the following five properties: 
\begin{enumerate}
\item The density of $f_i$ is $\alpha$. 
\item The only values of $f_i$ are $0$, $(1-\eta(N_1-1))\alpha$, $(1+\eta)\alpha$, and $\frac{1}{\zeta}(1+\eta)\alpha$, where $\eta=\sqrt{\epsilon/3}$ and $\zeta=\lceil 2p/3 \rceil /p$. 
\item The density of points $x\in \mathbb{F}_p^{n_i}$ for which $f_i(x) \not = (1+\eta)\alpha$ is at most $p\epsilon^{1/2}+\sum_{2 \leq j \leq i} \mu_{j}$.
\item For each $d \in \mathbb{F}_p^{n_i} \setminus \{0\}$, the density of $3$-APs with common difference $d$ in $f_i$ is less than $(1-\epsilon)\alpha^3$. 
\item Let $z_i$ be the $3$-AP density with common difference zero for $f_i$, which is also the density of $f_i^3$. For each $i$, we have $z_i<\left(1+\frac{4}{3}\mu_i \right) \alpha^3$. 
\end{enumerate}
The last function $f_s$ is our desired function $f$. Note that the second property implies each $f_i$ has values in $[0,1]$ as $\frac{1}{\zeta} \leq 3/2$, $1+\eta \leq 4/3$, $N_1 =p^{n_1}=p^{m_1} \leq (3/\epsilon)^{-1/2}=\eta^{-1}$, and $\alpha \leq 1/2$. 

\vspace{0.1cm}
\noindent \textbf{Construction for level 1.} Consider the space $\mathbb{F}_p^{n_1}$ of dimension $n_1=m_{1}=\left\lfloor \frac{1}{2}\log_{p}(3/\epsilon)\right\rfloor$. It has cardinality $N_1=p^{n_1}$.

Define $f_{1}:\mathbb{F}_{p}^{n_{1}}\rightarrow[0,1]$ by $f_{1}(0)=\left(1-\eta(N_1-1)\right)\alpha$ and $f_{1}(x)=(1+\eta)\alpha$ for all $x\ne 0$. It is clear from the construction that the density of $f_{1}$ is
$$\frac{1}{N_1}\left(1-\eta(N_1-1)\right)\alpha+\frac{N_1-1}{N_1}(1+\eta)\alpha=\alpha$$ and its only values are $\left(1-\eta(N_1-1)\right)\alpha$ and $(1+\eta)\alpha$. Thus $f_1$ satisfies the first and second of the five desired properties mentioned above. 

If $d \in \mathbb{F}_p^{n_1} \setminus \{0\}$, then any element of $\mathbb{F}_p^{n_1}$ is in a fraction $\frac{3}{N_1}$ of the $3$-APs with common difference $d$. So a $\frac{3}{N_1}$ fraction of these arithmetic progressions contain $0$ and the remaining $1-\frac{3}{N_1}$ fraction of these arithmetic progressions do not contain $0$. Hence, the density of $3$-APs with common difference $d$ is 
$$\left(1-\frac{3}{N_1}\right)(1+\eta)^3 \alpha^3 +\frac{3}{N_1}((1+\eta)\alpha)^2 \left(1-\eta(N_1-1)\right)\alpha,$$
which, by simplifying, is equal to 
$$\left(\left(1-\frac{3}{N_1}\right)(1+\eta)+\frac{3}{N_1} \left(1-\eta(N_1-1)\right)\right)(1+\eta)^2 \alpha^3 = (1-2\eta)(1+\eta)^2 \alpha^3=(1-3\eta^2-2\eta^3)\alpha^3 < (1-\epsilon)\alpha^3.$$

The trivial (zero) common difference has $3$-AP density 
\[z_1=\left(1-\frac{1}{N_1}\right)\cdot(1+\eta)^{3}\alpha^{3}+\frac{1}{N_1} \cdot \left(1-\eta(N_1-1)\right)^3\alpha^3,\] 
which, by simplifying, equals
\[\left(1+3\eta^2(N_1-1)-\eta^3(N_1-1)(N_1-2)\right)\alpha^3,\]
which  is less than $(1+3\eta)\alpha^3=(1+\sqrt{3\epsilon})\alpha^3<(1+\mu_1)\alpha^3$ as $\eta \leq 1/(N_1-1)$. We have thus showed that $f_1$ satisfies the fourth and fifth desired properties. 

The value of $f_1$ is $(1+\eta)\alpha$ for all but a fraction
$\frac{1}{N_1}<p\epsilon^{1/2}$ of the points, showing the third desired property. 

\vspace{0.1cm}
\noindent \textbf{Construction for level $i$ for $2\le i\le s$.} At level
$i$, we have already constructed a function $f_{i-1}:\mathbb{F}_{p}^{n_{i-1}}\to[0,1]$ in the previous level with the five desired properties.
Let $G_{i}$ be the set of elements in $\mathbb{F}_{p}^{n_{i-1}}$
with $f_{i-1}$ value $(1+\eta)\alpha$. Let $H_{i}$ be any subset
of $G_{i}$ of size $\mu_{i}p^{n_{i-1}}$.


For each $x\in H_{i}$, choose $v(x)\in\mathbb{F}_{p}^{m_{i}}$
a nonzero vector so that 
\begin{itemize}
\item these vectors are all distinct, 
\item for any $d \in \mathbb{F}_p^{m_i} \setminus \{0\}$,  we have $\mathbb{E}_{x \in H_i} [h(d \cdot v(x))] \leq \zeta^3-\frac{1}{125}$, where $h:\mathbb{F}_p \rightarrow [0,1]$ and $\zeta=1-\phi=\lceil 2p/3 \rceil /p$ are defined in the end of Subsection \ref{relativelyfewsubs},
\item for distinct $a,b\in H_{i}$, vectors $v(a)$ and $v(b)$ are linearly
independent, and 
\item if $a,b,c\in H_{i}$ form a nontrivial $3$-AP, then vectors $v(a)$, $v(b)$, and
$v(c)$ are linearly independent. 
\end{itemize}
The existence of such a choice of $v$ is guaranteed by Lemma \ref{lem:Choice of random directions}
below.

Define $f_{i}:\mathbb{F}_{p}^{n_{i}}\to[0,1]$ so that for each $a\in\mathbb{F}_{p}^{n_{i}}$,
with $x$ the first $n_{i-1}$ coordinates of $a$ and $y\in\mathbb{F}_{p}^{m_{i}}$
the last $m_{i}$ coordinates of $a$, we have $f_{i}(a)=f_{i-1}(x)$
if $x\not\in H_{i}$, and otherwise let $f_{i}(a)=\frac{1}{\zeta}f_{i-1}(x)=\frac{1}{\zeta}(1+\eta)\alpha$
if $y\cdot v(x)\in I$ and $f_{i}(a)=0$ if $y\cdot v(x)\notin I$, where $I \subset \mathbb{F}_p$ is a fixed interval with $\zeta p$ elements as in the end of Subection \ref{relativelyfewsubs}. 
Here $y\cdot v(x)$ denotes the dot product of $y$ and $v(x)$, which
is an element of $\mathbb{F}_{p}$. Note that this implies that for
each $f_{i}$, the value of each point is either $0$, $(1-\eta(N_1-1))\alpha$, $(1+\eta)\alpha$,
or $\frac{1}{\zeta}(1+\eta)\alpha$. Thus, $f_i$ has the second of the five desired properties. The first property, that $f_i$ has average value $\alpha$, follows easily from the definition of $f_i$. 

Recall that when we have finished the construction for level $s$, we let $f=f_{s}$.

\subsubsection{The proof}

For $i\ge2$, we first assure that we can make the choice of $v(x)$
as specified in the above construction. 
\begin{lem}
\label{lem:Choice of random directions} For $i\ge2$, there is a function $v:\mathbb{F}_p^{n_{i-1}}\to\mathbb{F}_{p}^{m_{i}}$
such that 
\begin{itemize}
\item for any $d \in \mathbb{F}_p^{m_i} \setminus \{0\}$,  we have $\mathbb{E}_{x \in H_i} [h(d \cdot v(x))] \leq \zeta^3-\frac{1}{125}$, and 
\item for distinct $a,b,c \in \mathbb{F}_p^{n_{i-1}}$, the vectors $v(a)$, $v(b)$, $v(c)$ are linearly
independent. 
\end{itemize}
\end{lem}
\begin{proof}
For $x\in \mathbb{F}_p^{n_{i-1}}$, choose $v(x)$ a random vector in $\mathbb{F}_{p}^{m_{i}}$
uniformly and independently of other $x$. 

For each nonzero $d \in \mathbb{F}_p^{m_i}$, we have $h(d \cdot v(x))$ is a nonnegative random variable with values $h(b)$ for $b \in \mathbb{F}_p$ equally likely. So this nonnegative random variable has expected value $\mathbb{E}_b [h(b)] \leq \zeta^3-(\phi^2/2 - \phi^3) \leq  \zeta^3 - \frac{2}{125}$ and maximum value at most $4/5$. The random variable $\mathbb{E}_{x \in H_i}[h(d \cdot v(x))]$ is $\frac{1}{|H_i|}$ times the sum of these $|H_i|$ independent identically distributed random variables with values in $[0,4/5]$. Thus, by Hoeffding's inequality, the probability that  $\mathbb{E}_{x \in H_i}[h(d \cdot v(x))]$ is at least $1/125$ more than its average is at most 
$\exp(-2 |H_i| (1/((4/5)125))^2)=\exp(-|H_i|/5000)$. By the union bound, the probability that there exists a nonzero  $d \in \mathbb{F}_p^{m_i}$ for which $\mathbb{E}_{x \in H_i}[h(d \cdot v(x))]$ is at least $1/125$ more than its average is at most $(p^{m_i}-1)\exp(-|H_i|/5000)$. Since $|H_i| =\mu_ip^{n_{i-1}} = \sigma m_i = 10^4 m_i\log p$, this probability is at most $p^{-m_i} \leq 1/4$. 

For distinct $a,b,c \in \mathbb{F}_p^{n_{i-1}}$, the probability that $v(a)=0$ is $p^{-m_i}$, the probability that $v(a)$ and $v(b)$ are linearly dependent given that $v(a) \not = 0$ is $p^{1-m_i}$, and the probability that $v(a)$, $v(b)$, and $v(c)$ are linearly dependent given that $v(a)$ and $v(b)$ are linearly independent is $p^{2-m_i}$. Thus, by the union bound, the probability that there is a triple of distinct $a,b,c \in \mathbb{F}_p^{n_{i-1}}$ for which $v(a)$, $v(b)$, $v(c)$ are linearly dependent is at most 
$${p^{n_{i-1}} \choose 3}\left(p^{-m_i}+p^{1-m_i}+p^{2-m_i}\right) < p^{3n_{i-1}+2-m_i} \leq 1/p \leq 1/3.$$
The second to last inequality is equivalent to showing $m_{i} \geq  3n_{i-1}+3$, which easily folows by induction on $i$. The base case $i=2$ follows from the bounds $n_1 \leq \frac{1}{2}\log_p (3/\epsilon)$, $m_2>\epsilon^{-1/8}$ by Inequality (\ref{m2ineq}), and $\epsilon \leq 2^{-160}p^{-8}$. For $i>2$, from (\ref{miineq}), we have 
$m_i >p^{m_{i-1}} = p^{n_{i-1}-n_{i-2}}>p^{2n_{i-1}/3}>n_{i-1}$, where we used the induction hypothesis to obtain the bound
$n_{i-1} \geq m_{i-1} \geq 3n_{i-2}+3>3n_{i-2}$. 

By the union bound, the probability that there are distinct $a,b,c \in \mathbb{F}_p^{n_{i-1}}$ such that $v(a),v(b),v(c)$ are linearly dependent, or there is a nonzero $d \in \mathbb{F}_p^{m_i}$ such that $\mathbb{E}_{x \in H_i}[h(d \cdot v(x))] > \zeta^2-\frac{1}{125}$ is at most $1/4+1/3<1$. Hence, there exists a choice of $v$ that satisfies the desired conditions. 
\end{proof}
We now give the proof of the main theorem. 
\begin{proof}[Proof of Theorem \ref{thm:Lower bound for weighted set}.]

Recall that $z_{i}$ is the $3$-AP density of the trivial common difference in level $i$, and we showed $z_{1}<(1+\mu_1)\alpha^{3}$.

We need to prove that $f_i$ has the five desired properties, which we do by induction on $i$, and have already observed the first two follow easily from the definition. Recall that we have shown that $f_1$ has the desired properties. The induction hypothesis is that $f_{i-1}$ has the desired properties. 

We next prove that $f_i$ has the fifth desired property, namely $z_{i}<(1+\frac{4}{3}\mu_{i})\alpha^{3}$ for each $i$. 
We have already proved it for $i=1$. Observe that 
\begin{eqnarray}
z_{i} & = & z_{i-1}+\mu_{i}\left(\zeta\left(\frac{1}{\zeta}(1+\eta)\alpha\right)^{3}-\left((1+\eta)\alpha\right)^{3}\right)=z_{i-1}+\left(\frac{1}{\zeta^2}-1\right)\mu_{i}(1+\eta)^{3}\alpha^{3}\nonumber \\
 & \leq & z_{i-1}+\frac{5}{4}(1+\eta)^{3}\mu_{i}\alpha^{3} \leq z_{i-1}+1.3\mu_{i}\alpha^{3}.\label{eq:fgh}
\end{eqnarray}
For $i \geq 2$, substituting in the induction hypothesis and $\mu_{i-1}=\mu_{i}/90$ into (\ref{eq:fgh}),
we have 
\[
z_{i}<z_{i-1}+1.3\mu_{i}\alpha^{3}<\left(1+\frac{4}{3}\mu_{i-1}\right)\alpha^{3}+1.3\mu_{i}\alpha^{3}<\left(1+\frac{4}{3}\mu_i\right)\alpha^{3},
\]
which completes the induction proof of the fifth desired property for $f_i$. 

Let $\rho_{j}(d)$ be the density of $3$-APs with common difference $d$ of $f_j$. If a nonzero $d \in \mathbb{F}_p^{n_i}$ is zero in the first $n_{i-1}$ coordinates, then consider the last $m_i$ coordinates of $d$, which we denote by $d' \in \mathbb{F}_p^{m_i}$. As $d$ is zero in its first $n_{i-1}$ coordinates but $d$ is nonzero, $d'$ is nonzero. We have \begin{eqnarray*} \rho_i(d) & = & \rho_{i-1}(0)+\frac{|H_i|}{p^{n_{i-1}}}\left(\mathbb{E}_{x \in H_i} h(d' \cdot x) - \zeta^3\right) \zeta^{-3}(1+\eta)^3 \alpha^3 \\ & = & z_{i-1}+\mu_i
\left(\mathbb{E}_{x \in H_i} h(d' \cdot x) - \zeta^3\right) \zeta^{-3}(1+\eta)^3 \alpha^3 \\ & \leq & 
 z_{i-1}-\mu_i \cdot \frac{1}{125} \cdot \zeta^{-3}(1+\eta)^3 \alpha^3  \\ & \leq & 
 z_{i-1}-\mu_i \cdot \frac{1}{64} \cdot (1+\eta)^3 \alpha^3  \\ & < & 
z_{i-1}-\mu_i \cdot \frac{1}{64} \cdot \alpha^3 \\ & < & 
(1+\frac{4}{3}\mu_{i-1})\alpha^3-\mu_i \cdot \frac{1}{64} \cdot \alpha^3 \\ & = & \left(1-\frac{7}{96}\mu_{i-1}\right)\alpha^3
\\ & < & (1-\epsilon)\alpha^3,
\end{eqnarray*}
where the last equality uses $\mu_i =90\mu_{i-1}$, and the last inequality uses $\mu_{i-1} \geq \mu_1 = 90p\epsilon^{1/4}$ and $\epsilon \leq 2^{-160}p^{-8}$. 

If $d$ is nonzero in the first $\mathbb{F}_{p}^{n_{i-1}}$ coordinates,
we prove in Lemma \ref{lem:Stability} below that the density $\rho_{i}(d)=\rho_{i-1}(d)<(1-\epsilon)\alpha^{3}$. This shows that $f_i$ has the fourth desired property. 

We thus need to show that $f_i$ satisfies the third desired property, and we have shown this for $i=1$. Note that $f_i$ has a fraction $\mu_i$ more of its input than $f_{i-1}$ having value not equal to 
$(1+\eta)\alpha^3$. It follows that the fraction of input for $f_i$ having value not equal to $(1+\eta)\alpha^3$ is at most 
 $$p\epsilon^{1/2} + \sum_{2 \leq j \leq i} \mu_j = p\epsilon^{1/2}+\sum_{j=2}^{i}90^{j}p \epsilon^{1/4} < 2 \cdot 90^{i}p \epsilon^{1/4}  \leq 2 \cdot (8p\epsilon^{1/4})^{-1} \cdot p \epsilon^{1/4} = 1/4.$$   

This inequality ensures that for each step $i$, there is sufficient space to choose $H_i$. Indeed, we need a fraction $\mu_i$ of the subspaces to still have weight $(1+\eta)\alpha$, and this holds by the above inequality as $\mu_i < 1-1/4 = 3/4$.

We now estimate the dimension of our final space. Since $n=n_s \geq m_s$, we get from Inequalities (\ref{m2ineq}) and (\ref{miineq}) that $n$ is at least a tower of $p$'s of height $s-2$ with $m_2>\epsilon^{-1/8}\geq p$ on top, where $s = \left\lfloor \log_{90} (1/(8\epsilon^{1/4}p)) \right\rfloor$. Thus we get the dimension $n$ is at least a tower of $p$'s of height $\log_{90} (1/(8\epsilon^{1/4}p)) - 2 \geq \log_{90} (2^{17}\epsilon^{-1/8}) - 2 \geq  \frac{1}{52}\log (2/\epsilon)$, where we used $\epsilon \leq 2^{-160}p^{-8}$.
\end{proof}
To complete the proof, we now prove Lemma \ref{lem:Stability}. 
\begin{lem}
\label{lem:Stability} Let $2 \leq i \leq s$. For $d \in \mathbb{F}_p^{n_i}$, let $d^{*} \in \mathbb{F}_p^{n_{i-1}}$ be the first $n_{i-1}$ coordinates of $d$. If $d^{*}$ is nonzero,  then the density of $3$-APs with common difference $d$ in $f_i$ is the same as the density of $3$-APs with common difference $d^{*}$ in $f_{i-1}$. That is, if $d^{*}$ is nonzero, then $\rho_{i}(d)=\rho_{i-1}(d^{*})$.
\end{lem}
\begin{proof}
Since $d^{*}$ is nonzero, in any three-term arithmetic progression $a,b,c$ with common difference $d$, the restrictions of the three points to the first $n_{i-1}$ coordinates are distinct. Let $a^{*}$ be the first
$n_{i-1}$ coordinates of $a$. Similarly define
$b^{*},c^{*}$. Fix $a^{*}=a_{0},b^{*}=b_{0},c^{*}=c_{0}$ for any $3$-AP $(a_0,b_0,c_0)$ in $\mathbb{F}_p^{n_{i-1}}$ of common difference $d^{*}$,
and consider all the $3$-APs with common difference $d$ such that the
first $n_{i-1}$ coordinates of points in the $3$-AP
coincide with $a_{0},b_{0},c_{0}$. Since $a_{0},b_{0},c_{0}$ are distinct, $v(a_0),v(b_0),v(c_0)$ are linearly independent. Hence, we can change the basis and view $v(a_{0}),v(b_{0}),v(c_{0})$
as three basis vectors of $\mathbb{F}_{p}^{m_i}$. Denote this set of basis vectors $B$. Let $a',b',c',d'$ be the restriction of $a,b,c,d$
to the last $m_{i}$ coordinates. 

Let $a_{1},b_{1},c_{1}$
be any three fixed values in $\mathbb{F}_{p}$. Let $L=p^{m_i-3}$. We prove that the number of $3$-APs in $\mathbb{F}_p^{n_i}$ with common difference $d$, $a^{*}=a_{0}$
and $a'\cdot v(a_{0})=a_{1}$, $b^{*}=b_{0}$ and $b'\cdot v(b_0)=b_1$, $c^{*}=c_0$ and $c'\cdot v(c_0)=c_1$ is $L$. Since $B$ is a basis for $\mathbb{F}_p^{m_i}$, if $x,x'\in \mathbb{F}_p^{m_i}$ satisfy $x\cdot v=x'\cdot v$ for all $v\in B$ then by linearity $x\cdot v=x'\cdot v$ for all $v\in \mathbb{F}_p^{m_i}$, in which case $x=x'$. Since there are $p^{m_i}$ elements in $\mathbb{F}_p^{m_i}$, and $p^{m_i}$ possible tuples $(x\cdot v)_{v\in B}$, each tuple must appear exactly once. The $3$-APs with common difference $d$, $a^{*}=a_{0}$
and $a'\cdot v(a_{0})=a_{1}$, $b^{*}=b_{0}$ and $b'\cdot v(b_0)=b_1$, $c^{*}=c_0$ and $c'\cdot v(c_0)=c_1$ are given by triples $(a',a'+d',a'+2d')$ such that $a'\cdot v(a_{0})=a_{1}$, $a'\cdot v(b_0) = b_1-d'\cdot v(b_0)$, $a'\cdot v(c_0)=c_1-2d'\cdot v(c_0)$. Hence the number of such $3$-APs is equal to the number of $a'\in \mathbb{F}_p^{m_i}$ such that $a'\cdot v(a_{0})=a_{1}$, $a'\cdot v(b_0) = b_1-d'\cdot v(b_0)$, $a'\cdot v(c_0)=c_1-2d'\cdot v(c_0)$. There is exactly one element of $\mathbb{F}_p^{m_i}$ such that its dot product with each vector in the basis $B \supset \{v(a_0),v(b_0),v(c_0)\}$ is fixed to an arbitrary value. Since there are exactly $p^{m_i-3}=L$ ways to choose the value of $a'\cdot v$ for $v\in B \setminus \{v(a_0),v(b_0),v(c_0)\}$, there are exactly $L$ elements $a'\in \mathbb{F}_p^{m_i}$ such that $a'\cdot v(a_{0})=a_{1}$, $a'\cdot v(b_0) = b_1-d'\cdot v(b_0)$, $a'\cdot v(c_0)=c_1-2d'\cdot v(c_0)$. Hence, there are $L$ $3$-APs with common difference $d$, $a^{*}=a_{0}$
and $a'\cdot v(a_{0})=a_{1}$, $b^{*}=b_{0}$ and $b'\cdot v(b_0)=b_1$, $c^{*}=c_0$ and $c'\cdot v(c_0)=c_1$.

Let $I \subset \mathbb{F}_p$ be an interval of length $\lceil 2p/3 \rceil$ as in the end of Subsection \ref{relativelyfewsubs}. 
For $x \in \mathbb{F}_p^{n_{i-1}}$, define $i_{x}:\mathbb{F}_{p}\to\mathbb{R}$ as follows. If $x \in H_i$ and $m\notin I$, let 
$i_{x}(m)=-(1+\eta)\alpha$, if $x \in H_i$ and $m \in I$, let $i_{x}(m)=(\frac{p}{|I|}-1)(1+\eta)\alpha$, and if $x \not \in H_{i}$ and $m \in \mathbb{F}_p$,  
let $i_{x}(m)=0$. For each $x \in \mathbb{F}_p^{n_{i-1}}$, we have $\sum_{m\in\mathbb{F}_{p}}i_{x}(m)=0$.
By our definition, we have $f_{i}(a)=f_{i-1}(a^{*})+i_{a^{*}}(a'\cdot v(a^{*}))$.
The density of $3$-AP with common difference $d$, fixing $a^{*}=a_{0},b^{*}=b_{0},c^{*}=c_{0}$,
is

\begin{align*}
 & \frac{1}{p^{3}L}\left(\sum_{a_{1},b_{1},c_{1}\in\mathbb{F}_{p}}L\cdot(f_{i-1}(a_{0})+i_{a_{0}}(a_{1}))\cdot(f_{i-1}(b_{0})+i_{b_{0}}(b_{1}))\cdot(f_{i-1}(c_{0})+i_{c_{0}}(c_{1}))\right)\\
 & =\frac{1}{p^{3}}\sum_{a_{1},b_{1},c_{1}\in\mathbb{F}_{p}}\left((f_{i-1}(a_{0})+i_{a_{0}}(a_{1}))\cdot(f_{i-1}(b_{0})+i_{b_{0}}(b_{1}))\cdot(f_{i-1}(c_{0})+i_{c_{0}}(c_{1}))\right)\\
 & =\frac{1}{p^{3}}\left(\sum_{a_{1}\in\mathbb{F}_{p}}(f_{i-1}(a_{0})+i_{a_{0}}(a_{1}))\right)\left(\sum_{b_{1}\in\mathbb{F}_{p}}(f_{i-1}(b_{0})+i_{b_{0}}(b_{1}))\right)\left(\sum_{c_{1}\in\mathbb{F}_{p}}(f_{i-1}(c_{0})+i_{c_{0}}(c_{1}))\right)\\
 & =f_{i-1}(a_{0})f_{i-1}(b_{0})f_{i-1}(c_{0}).
\end{align*}

Hence,

\begin{align*}
\rho_{i}(d) & =\mathbb{E}_{a_{0},b_{0},c_{0}}\left[\frac{1}{p^{3}L}\left(\sum_{a_{1},b_{1},c_{1}\in\mathbb{F}_{p}}L\cdot(f_{i-1}(a_{0})+i_{a_{0}}(a_{1}))\cdot(f_{i-1}(b_{0})+i_{b_{0}}(b_{1}))\cdot(f_{i-1}(c_{0})+i_{c_{0}}(c_{1}))\right)\right]\\
 & =\mathbb{E}_{a_{0},b_{0},c_{0}}\left[f_{i-1}(a_{0})f_{i-1}(b_{0})f_{i-1}(c_{0})\right]=\rho_{i-1}(d^*),
\end{align*}
which completes the proof. 
\end{proof}

\vspace{0.2cm}
\noindent \textbf{Acknowledgements.} We would like to thank David Fox for an observation in the game of SET which led to the study of multidimensional cap sets and this paper. We are indebted to Yufei Zhao for many helpful comments, including observing that Schur's inequality simplifies the density increment argument. We would like to thank Ben Green for many helpful comments, including suggestions which simplify the section on weak regularity and counting. We are grateful to Terry Tao for enlightening discussions.

\end{document}